\documentclass[12pt,UKenglish]{amsart}
\usepackage[margin=25 mm,bottom=35mm,footskip=15mm,top=35mm]{geometry}
\usepackage{enumerate,url,amssymb,amsthm,amsmath,xstring,enumitem,hyperref,subcaption,floatrow,setspace,tikz}
\usepackage[UKenglish]{babel}
\usepackage[numbers]{natbib}
\usepackage{xcolor}
\usepackage{etoolbox} 

\newtheorem{thm}{Theorem}[section]

\theoremstyle{plain}
\newtheorem{lemma}[thm]{Lemma}
\newtheorem{prop}[thm]{Proposition}
\newtheorem{cor}[thm]{Corollary}

\newtheorem{conj}[thm]{Conjecture}

\newtheorem{q}[thm]{Question}

\theoremstyle{definition}
\newtheorem{defn}[thm]{Definition}
\newtheorem{example}[thm]{Example}

\theoremstyle{remark}
\newtheorem{rem}[thm]{Remark}

\AtBeginEnvironment{thm}{\setlength{\parskip}{0pt}}
\AtBeginEnvironment{lemma}{\setlength{\parskip}{0pt}}
\AtBeginEnvironment{prop}{\setlength{\parskip}{0pt}}
\AtBeginEnvironment{cor}{\setlength{\parskip}{0pt}}
\AtBeginEnvironment{conj}{\setlength{\parskip}{0pt}}
\AtBeginEnvironment{q}{\setlength{\parskip}{0pt}}
\AtBeginEnvironment{defn}{\setlength{\parskip}{0pt}}
\AtBeginEnvironment{example}{\setlength{\parskip}{0pt}}
\AtBeginEnvironment{rem}{\setlength{\parskip}{0pt}}
\AtBeginEnvironment{proof}{\setlength{\parskip}{0pt}}

\title{The Manickam-Mikl\'os-Singhi Property in Graphs and Hypergraphs}
\author{Adam D\v{z}avoronok}
\address{Department of Applied Mathematics, Charles University, Faculty of Mathematics and Physics}
\email{adam.dzavoronok@mff.cuni.cz}

\begin{document}
\begin{abstract}
This paper studies the Manickam--Mikl\'os--Singhi (MMS) property for graphs and hypergraphs. Using the structural characterisation of the $2$-uniform case, we construct new families of regular graphs with the MMS property. We then analyse the Erd\H{o}s--R\'enyi random graph model $\mathbf{G}(n,p)$ and identify regimes in which the MMS property holds with high probability. Finally, we extend the matching-based sufficient condition to higher uniformities via pseudo-matchings and introduce a blowout construction that produces higher-uniformity hypergraphs with the MMS property from lower-uniformity examples.
\end{abstract}
\raggedbottom 
\maketitle

\section{Introduction}

Consider $n$ real numbers $x_1, x_2, \dots, x_n$ with a nonnegative sum. A natural question in extremal combinatorics asks: how many $k$-element subsets are guaranteed to also have a nonnegative sum? To motivate the lower bound for this problem, consider the weight distribution where one element is assigned a weight of $1$ and the remaining $n-1$ elements are assigned a weight of $-\frac{1}{n-1}$. The total sum is $0$, and the only $k$-element subsets with a nonnegative sum are those containing the element with weight $1$. Since there are exactly $\binom{n-1}{k-1}$ such subsets, this establishes a natural minimum. The Manickam--Mikl\'os--Singhi (MMS) conjecture \cite{manickam1987number, manickam1988first} asserts that this bound is tight for all sufficiently large $n$:

\begin{conj}[Manickam-Mikl\'os-Singhi]\label{conj:mms}
  Let $n,k$ be natural numbers such that $n \ge 4k$. Then, for any $n$ real numbers $x_1, x_2, \dots, x_n$ satisfying $\sum_{i=1}^n x_i \ge 0$, there exist at least $\binom{n-1}{k-1}$ subsets $A \subseteq [n]$ such that $|A|=k$ and $\sum_{i \in A} x_i \ge 0$.
\end{conj}

Some lower bound on $n$ is necessary. For example, when $k=2m+1$ and $n=4m$, set $x_1=x_2=\dots=x_{2m}=1$ and $x_{2m+1}=\dots=x_{4m}=-1$. By symmetry, there are as many $k$-subsets with positive sum as with negative sum; hence there are $\frac{1}{2}\binom{4m}{2m+1}$ nonnegative subsets, which is less than $\binom{4m-1}{2m}=\frac{2m+1}{4m}\binom{4m}{2m}$.

A second example occurs when $n=3k+1$. Set $x_1=x_2=x_3=2-3k$ and $x_4=\dots=x_{3k+1}=3$. Then a $k$-subset has nonnegative sum if and only if it contains none of $x_1,x_2,x_3$. Thus the number of nonnegative sets is $\binom{3k-2}{k}$, which is less than $\binom{3k}{k-1}$ for $k>2$.

A notable special case of Conjecture~\ref{conj:mms}, proved by Manickam and Singhi \cite{manickam1988first}, occurs when $k$ divides $n$:

\begin{thm}\label{thm:baranyai}
  Given $n$ reals $x_1, x_2, \dots, x_n$ with a nonnegative sum and a positive integer $k$, such that $k \mid n$, there exist at least $\binom{n-1}{k-1}$ $k$-sets with a nonnegative sum.
\end{thm}

We outline the proof of Theorem~\ref{thm:baranyai}. The main ingredient is Theorem~\ref{thm:baranyai-decomposition}.
\begin{thm}[Baranyai, 1975, \cite{baranyai1975factorization}]\label{thm:baranyai-decomposition}
  Let $n,k$ be positive integers with $k\mid n$. Then the complete $k$-uniform hypergraph $K_n^k$ can be decomposed into $\binom{n-1}{k-1}$ pairwise edge-disjoint perfect matchings.
\end{thm}
\noindent Each perfect matching contains at least one nonnegative edge, since the total sum of the vertex weights is nonnegative. Hence there are at least $\binom{n-1}{k-1}$ nonnegative edges.

A significant amount of research has focused on finding bounds for $m(k)$, the minimum value such that the conjecture holds for all $n \geq m(k)$. Historical progress on $m(k)$ improved the initial exponential upper bounds of Manickam and Mikl\'os $O(k^k)$ \cite{manickam1987number} and Tyomkyn $O((c \log k)^k)$ \cite{tyomkyn2012improved} to polynomial bounds by Alon, Huang, and Sudakov $33k^2$ \cite{ALON2012784} and Chowdhury, Sarkis, and Shahriari $8k^2$ \cite{chowdhury2014manickam}. Ultimately, a linear bound of $m(k) \leq 10^{46}k$ was established by Pokrovskiy~\cite{POKROVSKIY2015280}.

Motivated by this conjecture, a broader framework was established. We say that a general hypergraph $H$ possesses the \emph{MMS property} if every weighting function $f: V(H) \to \mathbb{R}$ satisfying $\sum_{v \in V(H)} f(v) \ge 0$ induces at least $\delta(H)$ \emph{nonnegative} edges, meaning edges $e$ for which $\sum_{v \in e} f(v) \ge 0$. Again we see that one cannot hope for more than $\delta(H)$ if we consider labelling in which one element is assigned a weight of $1$ and the remaining $(n-1)$ elements are assigned a weight of $-\frac{1}{n-1}$. Providing a full description of all hypergraphs with the MMS property extends far beyond the original MMS conjecture, which is restricted to determining when the complete $k$-uniform hypergraph $K_n^k$ has the MMS property.

Pokrovskiy~\cite{POKROVSKIY2015280} proposed the problem of describing all hypergraphs satisfying the MMS property, or at least identifying large, non-trivial families of them. Through his arguments, one can relate the general MMS property back to the regular MMS conjecture via the following lemma.
\begin{prop}(\cite{POKROVSKIY2015280})\label{prop:pokrovskiy}
If there exists a $d$-regular $k$-uniform hypergraph on $n$ vertices with the MMS property, then $K_n^k$ has the MMS property. 
\end{prop}
Prior progress on this generalized property includes work by Huang and Sudakov \cite{huang2013minimum}, who showed that for $n > 10k^3$, hypergraphs possessing equal codegree for every pair of vertices exhibit the MMS property.

\subsection*{Our Contributions}

Our main contributions are threefold. We begin with the graph case. A structural characterisation of $2$-uniform graphs with the MMS property was given by Kir\'aly, Kulkarni, McMeeking, and Mundinger~\cite{kiraly2018manickammiklossinghiparametergraphsdegree}:

\begin{thm}\label{thm:graph-characterisation}
  A graph $G=(V,E)$ has the MMS property if and only if, for every $E' \subseteq E$ such that $|E'| \le \delta(G)-1$, every independent set $A \subseteq V$ in $G-E'$ has a matching in $G-E'$ that covers every vertex of $A$.
\end{thm}
For completeness, we provide a proof of Theorem~\ref{thm:graph-characterisation} in Section~\ref{sec:mms-graphs}. Using this criterion, for sufficiently large number of vertices and some large enough regularity $d$, we construct connected $d$-regular graphs on $n$ vertices with the MMS property.

\begin{thm}\label{thm:regular-existence}
  For every odd  $n$ with Euler's totient function $\phi(n) \ge 8$, and for every even  $4 \le d \le \phi(n)$, there exists a connected $d$-regular graph on $n$ vertices with the MMS property.
\end{thm}

Our second contribution concerns the Erd\H{o}s--R\'enyi random graph model. To motivate the study, we note that the MMS property is highly sensitive to local structure and is therefore non-hereditary. For instance, a $d$-regular bipartite graph has the MMS property because it contains $d$ edge-disjoint perfect matchings, but a local modification may destroy the property. The operation illustrated in Figure~\ref{fig:local-modification} deletes two edges incident to a single vertex and rewires them so that one partition class can no longer be covered by a matching, contradicting the criterion in Theorem~\ref{thm:graph-characterisation}.

\begin{figure}[ht]
\centering
\begin{tikzpicture}[yscale=0.5, vertex/.style={circle, fill=black, inner sep=0pt, minimum size=5pt}, edge/.style={line width=1.5pt}]

\node at (-0.5, 3.2) {\Large \textit{d=3}};

\node[vertex] (y1) at (0, 2) {};
\node[vertex] (y2) at (1, 2) {};
\node[vertex] (y3) at (2, 2) {};
\node[vertex] (x1) at (0, 0) {};
\node[vertex] (x2) at (1, 0) {};
\node[vertex] (x3) at (2, 0) {};

\node[above=4pt] at (y2.north) {\Large $Y$};
\node[below=4pt] at (x2.south) {\Large $X$};

\foreach \top in {y1, y2, y3} {
  \foreach \bot in {x1, x2, x3} {
    \draw[edge] (\top) -- (\bot);
  }
}

\draw[->, line width=2.5pt] (2.8, 1) -- (4.2, 1);

\node[vertex] (u) at (5, 2) {};
\node[vertex] (t2) at (6, 2) {};
\node[vertex] (t3) at (7, 2) {};
\node[vertex] (b1) at (5, 0) {};
\node[vertex] (v) at (6, 0) {};
\node[vertex] (w) at (7, 0) {};

\node[right=2pt] at (u.east) {\Large $u$};
\node[below=6pt] at (v.south) {\Large $v$};
\node[below=6pt] at (w.south) {\Large $w$};

\draw[edge] (5, 4.2) ellipse (1.1cm and 0.9cm);
\node at (5, 4.2) {\Large $H$};
\draw[edge] (u.north) -- (4.2, 3.6);
\draw[edge] (u.north) -- (5.8, 3.6);

\draw[edge, red] (u) -- (b1);

\foreach \bot in {b1, v, w} {
  \draw[edge] (t2) -- (\bot);
  \draw[edge] (t3) -- (\bot);
}

\draw[edge, color=blue!70!cyan] (v) -- (w);

\end{tikzpicture}
\caption{A local modification that destroys the MMS property.}
\label{fig:local-modification}
\end{figure}
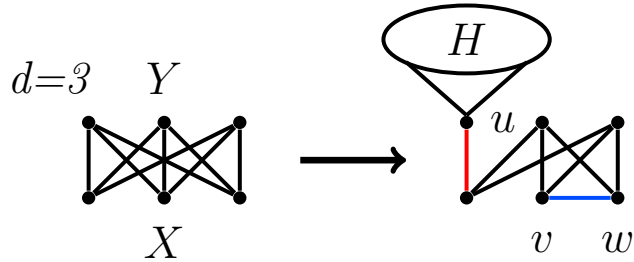

For the \(\mathbf{G}(n,p)\) model, we prove the following asymptotic result.

\begin{thm}\label{thm:random-graphs}
  Let $\epsilon>0$. For $p = \omega\left({\frac{\log n}{n}}\right)$ or $p<\frac{(1-\epsilon)\log n}{n}$, the random graph $\mathbf{G}(n,p)$ has the MMS property with high probability.
\end{thm}

For some of the probability ranges, this result also follows from the deep structural results of Krivelevich and Samotij \cite{krivelevich2012optimal}, and of Knox K\"uhn, and Osthus \cite{knox2015edge}.
\begin{thm}[Knox, K\"{u}hn, Osthus \cite{knox2015edge}]
  If $\frac{\log^{50}n}n\leq p\leq1-n^{-1/4}\log^9n$, then in $\mathbf{G}(n,p)$ there are with high probability $\lfloor\delta(G)/2\rfloor$ edge-disjoint Hamilton cycles. Furthermore, if $\delta(G)$ is odd, there is an additional edge-disjoint matching of size $\lfloor n/2\rfloor$.
\end{thm}
\begin{thm}[Krivelevich, Samotij \cite{krivelevich2012optimal}]
  There exists a positive constant $\epsilon$ such that the following holds. If
  $\frac{\log n}{n}\leq p\leq n^{-1+\epsilon}$ and $G\sim \mathbf{G}(n, p)$, then $G$ contains a collection of $\lfloor \delta(G)/2\rfloor$ edge-disjoint Hamilton cycles with high probability.
\end{thm}

When $n$ and $\delta(G)$ are even, each Hamilton cycle yields at least two nonnegative edges, since it is the union of two perfect matchings. Our proof of Theorem~\ref{thm:random-graphs} is more elementary and avoids this machinery. The case $p<\frac{(1-\epsilon)\log n}{n}$ follows from the standard threshold for isolated vertices in $\mathbf{G}(n,p)$.

Finally, we study the MMS property in higher uniformities. Motivated by the graph case, we introduce pseudo-matchings and prove a general sufficient condition for $k$-uniform hypergraphs. We then define a blowout construction that preserves the MMS property while increasing the uniformity. This construction gives a flexible source of nontrivial hypergraphs with the MMS property and leads to an extremal set-partition problem governing how many such blowouts can be combined edge-disjointly. The precise construction and the resulting extremal questions appear in Section~\ref{sec:high-uniformities}.
\subsection*{Acknowledgements}
 I would like to thank Mykhaylo Tyomkyn for many helpful discussions and comments while supervising my bachelor's thesis, from which this paper originated. This work was supported by GA\v{C}R grant 25-17377S.

\section{The MMS property in graphs}\label{sec:mms-graphs}
We start with the proof of Theorem~\ref{thm:graph-characterisation}.
\begin{proof}[Proof of Theorem~\ref{thm:graph-characterisation}]
 Let $G$ be a graph with the MMS property.
 Suppose for contradiction that there exist $E'\subseteq E$ and an independent set $A$ in $G-E'$ without a matching covering it.
 We will consider the following function $f:V(G)\rightarrow \mathbb{R}$.
 
  In $G-E'$, we take the subset $S\subseteq A$ for which $|N(S)|<|S|$, which exists by the failure of Hall's condition.
 If $|N(S)|>0$, we define $f$ by
 \[f(v)=\begin{cases}
   1 & \text{if } v\in S, \\
   -\dfrac{|S|}{|N(S)|}+\dfrac{\epsilon}{|N(S)|} & \text{if } v\in N(S), \\
   0 & \text{if } v\in A\setminus S, \\
   -\dfrac{\epsilon}{|V|} & \text{otherwise.}
   \end{cases}\]
where we take $0<\epsilon<\frac{|S|}{|N(S)\setminus A|}-1$.
 We verify that
   \begin{align*}     
   \sum_{v\in V}f(v)&=\sum_{v\in S}f(v)+\sum_{v\in N(S)}f(v)+\sum_{v\in A\setminus S }f(v)+\sum_{v\in V\setminus(A\cup N(S))}f(v)   \\ 
   &\ge|S|-|S|+\epsilon+0-\epsilon=0,
   \end{align*}
   If $|N(S)|=0$, then $f(v)=1$ for $v\in S$ and $-\frac{1}{|V|}$ otherwise, which also gives a total nonnegative sum.
 Furthermore, observe that in $G-E'$ every edge is negative.
 This means that in $G$ only the edges in $E'$ could be nonnegative.
 Therefore, in total, there are at most $|E'|=\delta(G)-1$ nonnegative edges, which contradicts the MMS property.
 For the converse, given a graph $G$ with the described edge property and a function $f$ satisfying $\sum_{v\in V} f(v)\ge0$, we will show that there are at least $\delta(G)$ nonnegative edges.
 Let $A:=\{v\in V : f(v)\ge 0\}$. If $|E(G[A])|\ge \delta(G)$, then $G$ has the MMS property. Suppose otherwise.
 By assumption, there exists a matching between $A$ and $N(A)$ in $G-E(G[A])$ covering each vertex of $A$.
 We claim that in the matching $M$ there has to be at least one nonnegative edge, because
 \begin{align*}
 \sum_{uv\in M}\bigl(f(u)+f(v)\bigr)=\sum_{v\in A}f(v)+\sum_{u\in V(M)\setminus A}f(u)&\\
 \ge \sum_{v\in A}f(v)+\sum_{u\in V\setminus A}f(u)=\sum_{v\in V}f(v)\ge0,
   \end{align*}
   where we used that every vertex in $V\setminus A$ has negative weight.
 If $|E(G[A])|=\delta(G)-1$, by including this edge from the matching, we obtain the MMS property, but if not, we can take this edge $e$ and still have a matching between $A$ and $N(A)$ in $G-\bigl(E(G[A])\cup\{e\}\bigr)$ by the assumption.
 This grants us another nonnegative edge. We can proceed this way $\delta-|E(G[A])|-1$ times, giving us $\delta$ nonnegative edges altogether.
 Since $f$ was given, this gives us the MMS property for $G$.
 \end{proof}
 We next show that many graphs from the following family have the MMS property, which will be our building blocks for the proof of Theorem \ref{thm:regular-existence}.
 \begin{defn}
   Let $n$ and $s_1,\dots,s_k$ be pairwise distinct positive integers satisfying $s_i<n$.
 We define the circulant graph $C_n^{s_1,s_2,\dots,s_k}$ as the graph on vertex set $\mathbb{Z}_n$ with an edge between $i$ and $j$ if
\[
i-j \pmod n \in \{\pm s_1,\pm s_2,\dots,\pm s_k\}.
\]
 Moreover, if $\gcd(n,s_i)=1$ for every $i$ and $n\nmid (s_i+s_j)$ for all distinct $i,j$, we say that the circulant graph is \textit{coprime}.
 \end{defn}
 
Even cycles have the MMS property, whereas odd cycles do not.
 On the other hand, for $k\ge2$ and $n>5$ there are coprime circulant graphs on an odd number of vertices that have the MMS property.
 Coprime circulant graphs are unions of edge-disjoint cycles spanning the same vertex set.
 This suggests a strategy for proving the MMS property: we consider the individual cycles rather than the entire graph.
 We therefore only need to treat the case when the graph has an odd number of vertices;
 in the even case, we immediately obtain the MMS property.
 We next show that if a cycle has one nonnegative edge, then the nonnegative vertices have to be arranged in a unique way.
 \begin{lemma}\label{lem:unique-cycle}
  Suppose $f:V(C_{2m+1})\rightarrow \mathbb{R}$ is such that $\sum_v f(v)\ge0$ and there is exactly one edge $uv\in E(C_{2m+1})$ satisfying $f(u)+f(v)\ge0$.
 Then $f(u),f(v)\ge0$ and furthermore, for each $xy\in E(C_{2m+1})\setminus\{uv\}$, the values $f(x)$ and $f(y)$ cannot be both nonnegative or both negative.
 \end{lemma}
\begin{proof}
We first show that there have to be exactly $m+1$ nonnegative vertices in this case.
 From this claim, the conclusion of the lemma follows immediately.
 Indeed, if there were more than $m+1$ nonnegative vertices, then we would have two pairs of consecutive nonnegative vertices on the cycle, producing at least two nonnegative edges.
 Now consider the case with fewer than $m+1$ nonnegative vertices.
 We can have either one or no edge with both endpoints nonnegative.
 In each subcase, we will show that there are one or two edge-disjoint matchings covering the nonnegative vertices, which, as argued, yields at least one more nonnegative edge.
 
 In the first subcase, removing the edge results in a path.
 We obtain the desired matching by greedily choosing edges from each side at a time.
 This is possible since the number of nonnegative vertices is less than that of the negative vertices.
 Similarly, in the second subcase, we can obtain the two disjoint matchings by greedily selecting edges adjacent to nonnegative vertices in the clockwise direction and then in the anticlockwise direction, while starting in both procedures from the same vertex, which completes the proof.
 \end{proof}
This yields the following theorem on the MMS property and coprime circulant graphs.
 \begin{thm}\label{thm:circulant}
A coprime circulant graph $C_{2m+1}^{t_1,t_2,\dots,t_k}$ has the MMS property if and only if for each $i\in [k]$ the set
\[\{ |t_i^{-1}t_1|, |t_i^{-1}t_2|,\dots, |t_i^{-1}t_k|\}\pmod{2m+1}\not\subseteq\{1,3\},\]
where $|\ell|=\min(\ell,2m+1-\ell)$ for $\ell\in\mathbb{Z}_{2m+1}$.
 \end{thm}
\begin{proof}

 Let $C_{2m+1}^{t_1,t_2\dots ,t_k}$ be a coprime circulant graph satisfying the described condition. We show it has the MMS property.
 First, we observe that each cycle gives us at least one nonnegative edge.
 On the other hand, if each cycle always has $2$ nonnegative edges, then the coprime circulant graph has the MMS property.
 Hence, we can assume that there exists a function and a cycle with just $1$ nonnegative edge.
 Without loss of generality, assume that this is the cycle generated by $t_1$.
 Consider the relabelling
\[
\varphi : \mathbb{Z}_{2m+1} \to \mathbb{Z}_{2m+1}, \qquad \varphi(x)=t_1^{-1}x.
\]
Since $\gcd(t_1,2m+1)=1$, the element $t_1$ is invertible modulo $2m+1$, so
$\varphi$ is a permutation of the vertex set.
 If $\{x,y\}$ is an edge in the cycle
generated by $t_i$, then
\[
x-y \equiv \pm t_i \pmod{2m+1},
\]
and therefore
\[
\varphi(x)-\varphi(y)=t_1^{-1}(x-y)\equiv \pm t_1^{-1}t_i \pmod{2m+1}.
 \]
Thus, after relabelling, we obtain the graph
\(
C_{2m+1}^{1,t_1^{-1}t_2,\dots,t_1^{-1}t_k},
\)
and the cycle with exactly one nonnegative edge is now the cycle generated by $1$.
 Transporting the weight function along $\varphi$ preserves the number of
nonnegative edges in each cycle.
 We may therefore assume that the cycle generated by $1$ has exactly one
nonnegative edge.
 We fix an order of vertices, such that the vertices $u,v$ from Lemma~\ref{lem:unique-cycle} are $1$ and $2m+1$, then by Lemma~\ref{lem:unique-cycle} the nonnegative vertices are exactly the odd numbers.
 We can now consider the other cycle generated by some $t_1^{-1}t_i$.
 If $|t_1^{-1}t_i|=3$, then a direct verification shows that this cycle has exactly $2$ nonnegative edges given by $\{1,2m-1\}$ and $\{2m+1,3\}$.
 In the other cases, where $|t_1^{-1}t_i|\ge5$ or even, we have at least $3$ nonnegative edges as seen by $\{i,i+|t_1^{-1}t_i|\}$ if $|t_1^{-1}t_i|$ is even or $\{2m+2-i,2m+2-i+|t_1^{-1}t_i|\}$ in the odd case, where $i\in\{1,3,5\}$.
 So overall, for the described coprime circulant graphs, there is at least one cycle with three nonnegative edges and exactly one cycle with one nonnegative edge.
 This proves one direction. The converse follows from the same argument.
\end{proof}
We observe the following consequence of Theorem~\ref{thm:circulant}.
 \begin{cor}\label{cor:circulant}
  For $k\ge3$, the coprime-circulant graph $C_n^{s_1,s_2,\dots ,s_k}$ has the MMS property for any valid tuples $s_1,s_2,\dots ,s_k$ and $n$.
 \end{cor}
Now we may finish the proof of Theorem \ref{thm:regular-existence}
\begin{proof}[Proof of Theorem~\ref{thm:regular-existence}]
  Write $d=2r$. Since $d\le \phi(n)$, we have $r\le \phi(n)/2$. The units of $\mathbb{Z}_n$ split into $\phi(n)/2$ pairs $\{\pm a\}$. Hence we may choose $r$ pairwise non-opposite units, one of which is $1$.

  If $r\ge 3$, let $s_1,\dots,s_r$ be such a choice with $s_1=1$. Then $C_n^{s_1,\dots,s_r}$ is connected, since it contains the cycle generated by $1$, and it is $2r=d$ regular. By Corollary~\ref{cor:circulant}, it has the MMS property. It remains to consider $r=2$, that is, $d=4$. Since $\phi(n)\ge 8$, among the pairs $\{\pm a\}$ of units we can choose a unit $k$ such that
  \[
    |k|\neq 1,3
    \qquad\text{and}\qquad
    |k^{-1}|\neq 3.
  \]
  Then $C_n^{1,k}$ is connected, $4$-regular, and satisfies the condition of Theorem~\ref{thm:circulant}. Hence it has the MMS property.
\end{proof}

\section{The MMS property in random graphs}
As noted in the Introduction, the case $p<\frac{(1-\epsilon)\log n}{n}$ follows from the standard threshold for isolated vertices in $\mathbf{G}(n,p)$: an isolated vertex has degree $0$, so the MMS property holds automatically. Hence, for the rest of the section we assume $p=\omega\left({\frac{\log n}{n}}\right)$. We split the proof into dense and sparse regimes, first treating $p=\omega\left({\sqrt{\frac{\log n}{n}}}\right)$ and then the range $p=\omega\left({\frac{\log n}{n}}\right)$ with $p\leq n^{-1/3}$.
\subsection{Dense Regime}
For the reader's convenience, let us recall the following standard bound on the tail of a random variable.
 It will help us to show that if we take a \textit{large} subset of vertices, then with high probability, the induced subgraph $G[A]$ has more than $\delta(G)$ edges.
 \begin{thm}[Chernoff bound]
  Let $X$ be the sum of independent Bernoulli random variables (not necessarily with the same
probability).
 Let $\mu=\mathbb{E}(X)$. Then, for all $\epsilon>0$,
\[\mathbb{P}(X<(1-\epsilon)\mu)\le e^{\frac{-\epsilon^2\mu}{2}}\]
and
\[\mathbb{P}(X>(1+\epsilon)\mu)\le e^{\frac{-\epsilon^2\mu}{1+\epsilon}}\]
\end{thm}
We shall use the following generalisation of Hall's marriage theorem, whose immediate corollary we will use to show that \textit{small} subsets $A$ satisfy the matching condition from Theorem~\ref{thm:graph-characterisation}.
 \begin{thm}[Lebensold \cite{Lebensold1977}]\label{thm:lebensold}
Let \( G = (A \cup B, E) \) be a bipartite graph.
 Then \( G \) contains \( k \) edge-disjoint matchings, each of size $|A|$, if and only if the following inequality holds:
\[\forall S \subseteq A,\quad \sum_{v \in B} \min\left(k,\; |N(v) \cap S| \right) \geq k |S|\]
\end{thm}
\begin{cor}\label{cor:lebensold}
  Let $G$ be a graph with minimum degree $\delta$.
 Let us consider the subset $A\subseteq V$, such that $|A|\le\frac{\delta}{2}$.
 If $\Delta_A$ is the largest degree in the induced graph $G[A]$.
 Then there exist $\delta-\Delta_A$ edge-disjoint matchings in the bipartite graph between $A$ and $N(A)$ of size $|A|$.
 \end{cor}
\begin{proof}
  The minimum degree in the bipartite graph between $A$ and $N(A)$ is at least \[\delta-\Delta_A\ge\delta-|A|+1\ge\frac{\delta}{2}+1\ge|A|,\]
 where we used the fact that at most $|A|-1$ edges adjacent to $v$ could be induced in $G[A]$.
 Now we apply Theorem~\ref{thm:lebensold} for $k=\delta-\Delta_A$. For each $S\subseteq A$ and $v\in N(A)$, the following holds $|N(v) \cap S|\le|S|\le|A|\le k$, which implies $\min(k,|N(v) \cap S|)=|N(v) \cap S|$.
 The value $|N(v) \cap S|$ is equal to the degree of $v$ in the induced bipartite subgraph on vertices $S\cup N(A)$, which means
 \[\sum_{v\in N(A)}|N(v) \cap S|=\sum_{v\in N(A)}\deg(v)=\sum_{u\in S}\deg(u)\ge(\delta-\Delta_A)|S|\]
for all $S\subseteq A$. This gives the desired $(\delta-\Delta_A)$ edge-disjoint matchings.
 \end{proof}
A direct verification shows that having $(\delta-\Delta_A)$ edge-disjoint matchings is a stronger condition than having a matching after the deletion of $(\delta-\Delta_A-1)$ edges.
Now we state following well known estimate for the concentration of the minimal degree in $\textbf{G}(n,p)$.
 \begin{lemma}\label{lem:min-degree}[\cite{FriezeKaronski2015} (Theorem 3.5)]
  Given a real number $\epsilon>0$, then in $\mathbf{G}(n,p)$ with $p=\omega(\frac{\log n}{n})$, the minimum degree $\delta$ is in the interval $[(1-\epsilon)(n-1)p,(1+\epsilon)(n-1)p)]$, with high probability.
 Moreover, \[\mathbb{P}(\delta\notin[(1-\epsilon)(n-1)p,(1+\epsilon)(n-1)p])\le ne^{-Kpn}\] for some constant $K>0$ depending on $\epsilon$ .
\end{lemma}
We can finally prove Theorem~\ref{thm:random-graphs}.
 \begin{proof}[Proof of Theorem~\ref{thm:random-graphs} (Dense Case)]
Let $\delta(G)$ be the minimum degree in the sampled graph $G\sim G(n,p)$, and fix $\epsilon = 0.01$. We define a failure event $\text{Fail}_\delta$ as the case when 
\[\delta(G) \notin \Big[ (1-\epsilon)(n-1)p, (1+\epsilon)(n-1)p \Big].\]

We will verify that for each subset $A\subseteq V$, either $|E(G[A])|\ge (1+\epsilon)p(n-1)$ or there are at least $\delta(G)-|E(G[A])|$ edge-disjoint matchings between $A$ and $N(A)$. This ensures that for every $E' \subseteq E(G)$ such that $|E'|\le\delta(G)-1$, the given set $A$ is either not independent in $G-E'$ or admits a matching covering it. By Theorem~\ref{thm:graph-characterisation}, this implies that $G$ has the MMS property.
 
We classify a subset $A\subseteq V$ as \textit{large} if $|A|\ge \frac{(1-\epsilon)}{2}(n-1)p$ and \textit{small} otherwise. For a large set, the expected number of induced edges satisfies: 
\begin{align*}
  \mathbb{E}\big(|E(G[A])|\big) = p\binom{|A|}{2} &\ge (1+\epsilon)p(n-1)\\
  (1-o(1))\frac{|A|^2}{2} &\ge (1+\epsilon)(n-1)\\
  (1-o(1))\cdot\frac{p^2(n-1)(1-\epsilon)^2}{8(1+\epsilon)} &\ge 1.
\end{align*}
Evaluating the constant term for $\epsilon = 0.01$, we have $\frac{0.99^2}{8.08} \approx 0.121 > 0$. Thus, the last inequality holds for sufficiently large $n$ due to the fact that $p=\omega\left(\sqrt{\frac{\log n}{n}}\right)$. 

We define a failure event $\text{Fail}_A$ for a large set $A$ as occurring when $|E(G[A])| < (1+\epsilon)p(n-1)$. Note that, by the same calculation as above, the inequality
\[(1-\epsilon)p\binom{|A|}{2} \ge (1+\epsilon)p(n-1)\] 
holds for any large set $A$. We now apply the Chernoff bound for any large set $A$ to obtain the estimate:
\[\mathbb{P}(\text{Fail}_A) \le \mathbb{P}\left(|E(G[A])| < (1-\epsilon)p\binom{|A|}{2}\right) \le e^{-Cp^3n^2}\]
for some constant $C>0$.

For small sets $A$, we show that there are at least $\delta(G)-|E(G[A])|$ edge-disjoint matchings between $A$ and $N(A)$. We proceed by picking a vertex $v\in A$ with the maximum degree $\Delta_A$ in the induced subgraph $G[A]$. This selection guarantees at least $\Delta_A$ induced edges in $G[A]$. Assuming $\text{Fail}_\delta$ does not occur, the definition of a small set ensures $|A| < \frac{(1-\epsilon)}{2}(n-1)p \le \frac{\delta(G)}{2}$. Therefore, we can apply Corollary~\ref{cor:lebensold} to obtain at least $\delta(G)-\Delta_A \ge \delta(G)-|E(G[A])|$ edge-disjoint matchings, as required. Observe that this portion of the argument is purely deterministic. 

To conclude the proof, we show that the probability of any failure event occurring tends to zero. Observe that if $\text{Fail}_A$ occurs for a large set $A$, then the event $\text{Fail}_{A'}$ must also occur for all large subsets $A'\subseteq A$. Therefore, it suffices to only consider the failure events $\text{Fail}_A$ for sets exactly of size $|A|=\left\lceil \frac{(1-\epsilon)}{2}(n-1)p \right\rceil= m$. To limit the number of such sets, we use the estimate $\binom{n}{m} < n^{pn} = e^{pn \log n}$. 

Applying the union bound, we obtain:
\begin{align*}
  \mathbb{P}\left(\big(\bigcup_{|A|=m}\text{Fail}_A\big)\cup \text{Fail}_{\delta}\right) &\le \binom{n}{m}\mathbb{P}(\text{Fail}_A) + \mathbb{P}(\text{Fail}_\delta) \\
  &\le e^{pn \log n - C p^3n^2} + ne^{-Kpn} \\
  &\le e^{-pn \cdot\omega(\log n)} + ne^{-Kpn},
\end{align*}
where we used $p=\omega\left(\sqrt{\frac{\log n}{n}}\right)$. As $n\rightarrow\infty$, the right-hand side tends to $0$.
 \end{proof}
\subsection{Sparse Regime}
We now consider the range $\omega\left(\frac{\log n}{n}\right) \le p \le n^{-1/3}$. In addition to the lemmas used in the dense case, we need the following standard bound on the independence number.

\begin{lemma} \label{lem:independence-number}[\cite{FriezeKaronski2015} (Theorem 7.3)]
Let $G \sim \mathbf{G}(n,p)$ with $p = \omega\left(\frac{\log n}{n}\right)$. Then, with high probability, the independence number $\alpha(G)$ satisfies $\alpha(G) \le \frac{4 \log n}{p}$.
\end{lemma}

Unlike in the dense case, the sparse argument uses the structural criterion in Theorem~\ref{thm:graph-characterisation} more directly. Given a set $E'\subseteq E$ with $|E'|\leq\delta(G)-1$, we consider independent sets in $G-E'$, bound their size, and use expansion in $\mathbf{G}(n,p)$ to verify Hall's condition.
\begin{proof}[Proof of Theorem~\ref{thm:random-graphs} (Sparse Case)]
Let $\epsilon = 0.01$. We condition on two properties that hold with high probability for $G \sim \mathbf{G}(n,p)$ in this regime. First, by Lemma~\ref{lem:min-degree}, the minimum degree satisfies $\delta(G) \in [(1-\epsilon)np, (1+\epsilon)np]$. Second, the independence number is bounded by $\alpha(G) \le \frac{4 \log n}{p}$ by Lemma~\ref{lem:independence-number}.

Suppose for contradiction that $G$ lacks the MMS property. By Theorem~\ref{thm:graph-characterisation}, there exists a subset of edges $E' \subseteq E(G)$ with $|E'| \le \delta(G) - 1$, and an independent set $A$ in $G - E'$ that admits no covering matching. By Hall's condition, there must exist a subset $S \subseteq A$ such that its external neighborhood in the deleted graph satisfies $|N_{G-E'}(S)| \le |S| - 1$, where $N_{G-E'}(S)$ is its open neighbourhood, which in this case coincides with its external neighbourhood.

Relating this back to the original graph $G$, the neighborhood can increase by at most the number of deleted edges $|E'|$. Thus, we have
\begin{equation}
  |N_G(S)| \le |N_{G-E'}(S)| + |E'| \le |S| - 1 + \delta(G) - 1 < |S| + (1+\epsilon)np. \label{eq:neighborhood-bound}
\end{equation}
Let $s = |S|$ and denote the number of external neighbors by $X_S = |N_G(S) \setminus S|$. Since $X_S \le |N_G(S)|$, the failure condition \eqref{eq:neighborhood-bound} requires $X_S < s + (1+\epsilon)np$. We may assume that $s\ge 2$. Indeed, if $S=\{v\}$, then $v$ has at least $\delta(G)$ neighbours in $G$. Since at most $\delta(G)-1$ edges are deleted, at least one edge incident to $v$ remains in $G-E'$. Thus $|N_{G-E'}(S)|\ge 1$, contradicting the Hall-violating inequality $|N_{G-E'}(S)|\le |S|-1=0$.

Next, we establish an absolute upper bound on $s$. Because $A$ is an independent set in $G - E'$, any edge connecting two vertices in $A$ within the original graph $G$ must belong to $E'$. By taking $A$ and removing at most one endpoint for every internal edge, we obtain a strict independent set $I$ in $G$. Thus, $|I| \ge |A| - |E'|$. Since $|I| \le \alpha(G)$, we have:
\begin{equation*}
  |A| \le \alpha(G) + |E'| \le \frac{4 \log n}{p} + (1+\epsilon)np.
\end{equation*}
Because $p = \omega\left(\frac{\log n}{n}\right)$, the term $\frac{4\log n}{p} = o(n)$ and $p\le n^{-1/3}$ gives us $(1+\epsilon)pn=o(n)$. Thus, we obtain the asymptotic upper bound $|A|= o(n)$ for large $n$. Since $S \subseteq A$, it suffices to consider sizes $2 \le s \le s_{max}=o(n)$, where $s_{max}$ is the size of the largest independent set in $G-E'$.
For a fixed set $S$ of size $s$, the probability that a vertex $v \notin S$ connects to $S$ is $q = 1 - (1-p)^s$. We divide our analysis into two regimes based on the size of $sp$.

\textbf{Regime 1 ($2 \le s \le 0.1/p$):} Since $sp \le 0.1$, we can use the Taylor expansion bound $q \ge sp(1 - \frac{sp}{2})$, which evaluates to $q \ge 0.95sp$. The expected number of external neighbors is $\mu = \mathbb{E}[X_S] = (n-s)q$. Because $s= o(n)$, we have $n-s \ge 0.99n$, yielding $\mu \ge 0.94snp$. 

We apply the Chernoff bound with the failure threshold $K = s + (1+\epsilon)np$. The relative deviation $\eta = 1 - \frac{K}{\mu}$ is bounded by considering:
\begin{equation*}
  \frac{K}{\mu} \le \frac{s + (1+\epsilon)np}{0.94 snp} = \frac{1}{0.94np} + \frac{1+\epsilon}{0.94s}.
\end{equation*}
For large $n$, $\frac{1}{0.94np} < 0.01$. Since $s \ge 2$, $\frac{1+\epsilon}{0.94s} \le 0.537$. Thus $\frac{K}{\mu} \le 0.55$, meaning $\eta \ge 0.45$. The Chernoff bound yields:
\begin{equation*}
  \mathbb{P}(X_S < s + (1+\epsilon)np) \le \exp\left(-\frac{0.45^2}{2}(0.94snp)\right) \le \exp(-0.09 snp).
\end{equation*}
To ensure no such set $S$ exists, we apply the union bound over all subsets in this regime:
\begin{equation*}
  \mathbb{P}(\text{Fail}_1) \le \sum_{s=2}^{\lfloor 0.1/p \rfloor} \binom{n}{s} \exp(-0.09 snp) \le \sum_{s=2}^{\infty} \exp\Big(s (\log n - 0.09 np)\Big).
\end{equation*}
Because $p = \omega\left(\frac{\log n}{n}\right)$, we have $0.09np \ge 2 \log n$ for large $n$. The sum is therefore bounded by a convergent geometric series $O(n^{-2}) = o(1)$.

\textbf{Regime 2 ($0.1/p < s \le s_{max}$):} In this regime, the neighborhood expands rapidly. We have $q = 1 - (1-p)^s \ge 1 - e^{-sp} \ge 1 - e^{-0.1} \ge 0.095$. The expectation becomes $\mu \ge (0.99n)(0.095) \ge 0.094n$.

The failure threshold is bounded by $K = s + (1+\epsilon)np$. Since $\mu = \Theta(n)$ and $K = o(n)$, the relative deviation $\eta = 1 - \frac{K}{\mu} \to 1$. Using $\eta \ge 0.9$, the Chernoff bound yields:
\begin{equation*}
  \mathbb{P}(X_S < s + (1+\epsilon)np) \le \exp\left(-\frac{0.9^2}{2} (0.094n)\right) \le \exp(-0.03n).
\end{equation*}
We union bound over all possible sets of size $s \le s_{max}$. Utilizing the standard estimate $\binom{n}{k} \le \left(\frac{en}{k}\right)^k$, the total number of sets is bounded by:
\begin{equation*}
  \sum_{s=1}^{s_{max}} \binom{n}{s} \le s_{max} \binom{n}{s_{max}} \le \exp\left(s_{max} \log \left(\frac{en}{s_{max}}\right)\right).
\end{equation*}
Since $s_{max}=o(n)$ then the term $s_{max}\log(\frac{ne}{s_{max}})=o(n)$. The final union bound becomes:
\begin{equation*}
  \mathbb{P}(\text{Fail}_2) \le \exp\Big(o(n) - 0.03n \Big)=o(1).
\end{equation*}

Because the failure probabilities in both regimes vanish as $n \to \infty$, Hall's condition holds globally in the graph $G-E'$. Consequently, $\mathbf{G}(n,p)$ possesses the MMS property with high probability.
\end{proof}
Another natural extremal question to ask is whether there exists a function $d(n)$ such that each regular graph on $n$ vertices  with a degree of at least $d(n)$ has the MMS property.
 This can be seen as the deterministic counterpart to $\mathbf{G}(n,p)$, that will be discussed in the next section. In the following proposition, we present a related result.
 \begin{prop}\label{prop:extremal-degree}
  For $n\ge8$ and $d\ge \lceil\frac{n}{2}\rceil+2$, every admissible $d$-regular graph on $n$-vertices has the MMS property.
 Furthermore, for every $k$, there exists a $2k+2$-regular graph on $4k+1$-vertices without the MMS property.
 \end{prop}
\begin{proof}
  Let us start with the upper bound.
 Set $d\ge\lceil\frac{n}{2}\rceil+2$ and let $G$ be a $d$-regular graph on $n$ vertices.
 Suppose for contradiction that there exist $E'$ and $A$ as in Theorem~\ref{thm:graph-characterisation}.
 First, observe that $|A|-|V\setminus A|<2$ otherwise, there are at least $d$-edges in $E[A]$. Therefore, $|A|\le\lceil\frac{n}{2}\rceil$.
 On the other hand, if $|A|\le\frac{d}{2}$, the result follows by Corollary~\ref{cor:lebensold}.
 Since $A$ does not have a matching covering it, we can assume that $\deg(v)\le|A|-1$ for $v\in A$ otherwise, we take the subset $S\subseteq A$ violating Hall's condition.
 For each such vertex, we delete at least $(d-|A|+1)$, and each edge could be counted at most twice, so 
  \[d-1\ge|E'|\ge \frac{|A|}{2}(d-|A|+1),\]
 where the minimum is achieved at the endpoints, that is, $|A|\in\{\frac{d}{2},\lceil\frac{n}{2}\rceil\}$.
 Substituting these endpoint values into the inequality gives a contradiction.
 \noindent For the lower bound construction, we take the following graphs $H=C_{2k+1}$ and $F=\{\text{matching on $2k$ vertices\}}$ and construct a complete bipartite graph between the vertices of $F$ and $H$, preserving the original edges to form a graph $G$.
 A direct verification shows that $G$ is $2k+2$-regular. Moreover, if we remove all $2k+1$ edges of the graph $H$, then its vertices form an independent set in $G$ without matching covering them since $|N(H)|=|V(F)|<|H|$, and thus the result follows by Theorem~\ref{thm:graph-characterisation}.
 \end{proof}

\section{Higher uniformities}\label{sec:high-uniformities}
In this section, we will build upon the established results and constructions for $2$-uniform graphs in the previous sections and apply them to higher uniformities.
 First, in the spirit of Theorem~\ref{thm:graph-characterisation}, we extend the sufficient condition for the MMS property to higher uniformities.
 For this purpose, we slightly modify the notion of a matching.
 \begin{defn}
Given a hypergraph $H=(V,E)$ and a set $A\subseteq V$, a \textit{pseudo-matching} saturating $A$ is a set $M\subseteq E$ such that any two distinct edges $e,f\in M$ satisfy $e\cap f\subseteq A$, and
\(
A\subseteq \bigcup_{e\in M} e.
\)
 \end{defn}
In other words, each vertex of $A$ is covered by the edges of the pseudo-matching at least once, and other vertices are covered at most once. We now state the sufficient condition, whose proof is analogous to the $2$-uniform case.
 \begin{thm}\label{thm:pseudo-matching-sufficient}
Let \(H=(V,E)\) be a \(k\)-uniform hypergraph with minimum degree \(\delta\).
 Suppose that for every set \(E'\subseteq E\) with \(|E'|\le \delta-1\), every independent set \(A\) in \(H-E'\) admits a pseudo-matching saturating \(A\) in \(H-E'\).
 Then \(H\) has the MMS property.
\end{thm}
\begin{proof}
   Let \(f:V(H)\to \mathbb R\) satisfy
\(
\sum_{v\in V(H)} f(v)\ge 0,
\)
and set
\(
A=\{v\in V(H): f(v)\ge 0\}.
\)
If \(|E(H[A])|\ge \delta\), then all edges of \(H[A]\) are nonnegative, so \(H\) has at least \(\delta\) nonnegative edges and the result follows.
 Assume now that \(|E(H[A])|\le \delta-1\), and put \(E_0=E(H[A])\). Then \(A\) is independent in \(H-E_0\).
 By the hypothesis, there exists a pseudo-matching \(M_1\) saturating \(A\) in \(H-E_0\).
 Since every vertex of \(A\) is covered by \(M_1\) at least once and every vertex of \(V\setminus A\) is covered at most once, we have
\[
\sum_{e\in M_1}\sum_{v\in e} f(v)
\ge
\sum_{v\in A} f(v)+\sum_{v\in V\setminus A} f(v)
=
\sum_{v\in V(H)} f(v)\ge 0.
\]
Hence, at least one edge \(e_1\in M_1\) is nonnegative.
 If \(|E(H[A])|=\delta-1\), then the \(\delta-1\) edges of \(H[A]\) together with \(e_1\) already give at least \(\delta\) nonnegative edges, so the result follows.
 Otherwise, set \(E_1=E_0\cup\{e_1\}\). Since \(|E_1|\le \delta-1\) and \(A\) is still independent in \(H-E_1\), the hypothesis again yields a pseudo-matching \(M_2\) saturating \(A\) in \(H-E_1\).
 By the same argument, \(M_2\) contains a nonnegative edge \(e_2\notin E_1\).
 Continuing inductively, after \(t\) steps we have a set
\[
E_t=E(H[A])\cup\{e_1,\dots,e_t\}
\]
with \(|E_t|\le \delta-1\), and \(A\) is independent in \(H-E_t\).
 Therefore, we can find another pseudo-matching saturating \(A\) in \(H-E_t\), yielding a new nonnegative edge \(e_{t+1}\notin E_t\).
 Repeating this process \(\delta-|E(H[A])|\) times, we obtain at least \(\delta\) distinct nonnegative edges in \(H\).
 Since \(f\) was arbitrary, \(H\) has the MMS property.
\end{proof}
As in the previous sections, we will construct some classes of hypergraphs with the MMS property.
 We will first describe rather simple families of hypergraphs with the MMS property, from which we will build further hypergraphs.
 Our stepping stone will be to use graphs of lower uniformities with the MMS property, as described in the following lemma.
 \begin{lemma}\label{lem:blowout}
  Let $H$ be a $k$-uniform hypergraph with the MMS property on vertex set $[n]$, and let its minimum degree be $\delta$.
  Define $H^{\bigoplus m}$ to be the $mk$-uniform hypergraph on vertex set $[mn]$ obtained as follows: each vertex $i\in [n]$ is replaced by the class
  \[
    V_i=\{i+jn:0\le j\le m-1\},
  \]
  and each edge $e\in E(H)$ is replaced by the edge
  \(
    \bigcup_{i\in e} V_i.
  \)
  Then $H^{\bigoplus m}$ has a minimum degree $\delta$ and has the MMS property.
\end{lemma}

\begin{proof}
Each vertex in the class $V_i$ lies precisely in the blown-up edges corresponding to edges of $H$ that contain $i$. Hence the degree of each vertex in $V_i$ is $\deg_H(i)$, so the minimum degree is preserved.

Let $f:V(H^{\bigoplus m})\to \mathbb R$ satisfy
\[
\sum_{v\in V(H^{\bigoplus m})} f(v)\ge 0.
\]
Define a weight function $g:V(H)\to\mathbb R$ by
\(
g(i)=\sum_{j=0}^{m-1} f(i+jn).
\)
Then
\[
\sum_{i=1}^n g(i)=\sum_{v\in V(H^{\bigoplus m})} f(v)\ge 0.
\]
Since $H$ has the MMS property, at least $\delta$ edges $e\in E(H)$ satisfy
\(
\sum_{i\in e} g(i)\ge 0.
\)
For each such edge, the corresponding blown-up edge in $H^{\bigoplus m}$ has $f$-weight
\[
\sum_{i\in e}\sum_{j=0}^{m-1} f(i+jn)=\sum_{i\in e}g(i)\ge0.
\]
Thus, $H^{\bigoplus m}$ has at least $\delta$ nonnegative edges and hence has the MMS property.
\end{proof}

We may refer to hypergraphs of the form $ H^{\bigoplus m}$ as \emph{blowouts}. In addition, if $H$ is regular, one may obtain the following result.
 \begin{prop}
 If $H$ is a $k$-uniform $d$-regular hypergraph on $n$ vertices with the MMS property, then for any positive integer $m$, the hypergraph $K_{mn}^{mk}$ has the MMS property.
 \end{prop}
 \begin{proof}
     Note that $H^{\bigoplus m}$ is $mk$-uniform and $d$-regular on $mn$ vertices. By applying Proposition~\ref{prop:pokrovskiy}, we obtain the desired result.
 \end{proof}
 This essentially reduces Conjecture~\ref{conj:mms} to cases when pairs $(n,k)$ are coprime.
 
Our goal is to construct MMS hypergraphs using the blowout operation. Our last ingredient for the construction of the MMS hypergraph is the following simple observation. It will be used when several blowouts are placed on the same vertex set.
\begin{lemma}\label{lem:edge-disjoint-union}
Let $F_1,\dots,F_t$ be pairwise edge-disjoint $r$-uniform hypergraphs on the same vertex set. Suppose that each $F_i$ is $d_i$-regular and has the MMS property. Then $F:=F_1\cup\cdots\cup F_t$ is $\left(\sum_{i=1}^t d_i\right)$-regular and has the MMS property.
\end{lemma}
\begin{proof}
Regularity follows from edge-disjointness and from the fact that all $F_i$ have the same vertex set. Let $f$ be any weighting of the common vertex set with nonnegative total sum. For every $i$, the MMS property of $F_i$ gives at least $d_i$ nonnegative edges in $F_i$. Since the edge sets of the $F_i$ are pairwise disjoint, these nonnegative edges are distinct in $F$. Therefore $F$ has at least $\sum_i d_i=\delta(F)$ nonnegative edges.
\end{proof}
 
We now use Lemma~\ref{lem:blowout} and Lemma~\ref{lem:edge-disjoint-union} to construct more families of hypergraphs with the MMS property. Our main goal is to establish an analogue of Theorem~\ref{thm:regular-existence}: under suitable divisibility conditions, there is a range of degrees $d$ for which one can construct $d$-regular $k$-uniform hypergraphs with the MMS property.  If $k\mid n$, then from Theorem~\ref{thm:baranyai}, for each positive integer $d\le\binom{n-1}{k-1}$, such hypergraphs exist by considering $d$ perfect matchings from the Baranyai system.
 
 For some of the remaining cases, our plan is to use the results for $2$-uniform graphs in the form of Theorem~\ref{thm:regular-existence} and Proposition~\ref{prop:extremal-degree} to establish a range of possible regularities. We will focus on the case where uniformity $k=2m$, where $m$ satisfies $m\mid n$. Now, if we consider a $d$-regular graph with the MMS property on $\frac nm$ vertices. Then, by applying Lemma~\ref{lem:blowout}, we obtain a $d$-regular $2m$-uniform hypergraph with the MMS property on $n$ vertices. Thus, the degree ranges obtained earlier for the graphs carry over to the hypergraphs with $\frac nm$ in place of $n$. The case $k=2m\nmid n$ is the one of interest here, as we have established possible degree ranges for graphs and the case where uniformity divides the number of vertices, well understood. However, it is possible to generalise these ideas. If we blowout $\ell$ uniform hypergraphs, we would consider parameters $m\mid n$ and $\ell m\nmid n$. 
 
 In the context of higher uniformities, these blowouted hypergraphs are rather sparse. Indeed, a blowout hypergraph $k\ell$, a uniform hypergraph from $k$, a uniform one on $n$ vertices, may have a maximal degree of at most $\binom{n-1}{k-1}$, whereas the maximal degree for the blowout is $\binom{n\ell -1}{k\ell -1}$.
 However, by modifying the original construction, we can significantly improve it.
 The idea is to take an edge-disjoint union of multiple hypergraphs $H_i^{\bigoplus m}$, where the original graphs $H_i$ need not be the same.
 Specifically, each $H_i$ is a graph with the MMS property on $n$ vertices.
 We then assign pairwise distinct partitions into $m$-tuples of $[mn]$ to each of the $H_i^{\bigoplus m}$, by which we replace the original vertices of $H_i$ to obtain $H_i^{\bigoplus m}$ as before.
 We demonstrate this construction in the following example.
\begin{example}
  We construct $4$-uniform $16$-regular hypergraphs on $22$ vertices with the MMS property.
 First, we consider the following partitions of the $22$ vertices into $11$ pairs with $\mathcal{P}_1=\{(2k-1,2k)\mid 1\leq k\leq 11\}$ and $\mathcal{P}_2=\{(k,k+11)\mid 1\le k\leq 11\}$.
 We choose $H_1=K_{11}$ and $H_2=C_{11}^{1,2,3}$, which have the MMS property.
 We construct the blowouts $H_1^{\bigoplus2}$ and $H_2^{\bigoplus2}$, where a vertex in $H_i$ is blown out to a pair in $\mathcal{P}_i$.
 We observe that the produced hypergraphs are edge-disjoint. We may then set our final hypergraph to be $H_1^{\bigoplus2}\cup H_2^{\bigoplus2}$. It is $16$-regular and has the MMS property by Lemma~\ref{lem:edge-disjoint-union}.
 \end{example}

In order to get bounds on the range of regularities we can obtain, we have to estimate the number of certain partitions $[mn]$ into $m$-tuples.
 A sufficient condition that these partitions have to satisfy is that if we assign the $m$-tuples to be vertices of $K_n^{\bigoplus m}$ for each partition, this will result in edge-disjoint hypergraphs.
 In such a case, we can then replace some copies of $K_n$ with other valid graphs and still maintain edge-disjointness.
 We state this problem as a standalone extremal question because it is an interesting problem in its own right.
 \begin{q}\label{q:partition}
  Let $k,\ell,m,n$ be positive integers and let $\mathcal{P}_1,\mathcal{P}_2,\dots,\mathcal{P}_\ell$ be pairwise distinct partitions of $[mn]$ into subsets of size $m$, such that for every two distinct indices $a,b\in[\ell]$ there are no $k$-element subfamilies $\{e_1,\dots,e_k\}\subseteq \mathcal{P}_a$ and $\{f_1,\dots,f_k\}\subseteq \mathcal{P}_b$ satisfying
  \[
  \bigcup_{i=1}^k e_i=\bigcup_{i=1}^k f_i.
  \]
 We call such a family of partitions conflictless. What is the largest possible size of a conflictless family as a function of the parameters $n,m,k$?
 Let us further denote this maximum value by $\operatorname{MaxM}(n,m,k)$.
\end{q}
The parameter $k$ in Question~\ref{q:partition} is the uniformity of the original hypergraphs before the blowout; in the graph blowouts below we use $k=2$.
 There is the simple upper bound
 \[
 \operatorname{MaxM}(n,m,k)\leq\frac{\binom{nm}{km}}{\binom{n}{k}},
 \]
 since each $km$-tuple can be expressed in at most one way as a union of $k$ disjoint $m$-tuples from some $\mathcal{P}_i$.
 If $m,k$ are fixed and $n\rightarrow\infty$, then $\frac{\binom{nm}{km}}{\binom{n}{k}}=\Theta(n^{k(m-1)})$. However, in our context of constructing certain families of hypergraphs with the MMS property, we need a lower bound for $\operatorname{MaxM}(n,m,k)$.
 As mentioned, we are particularly interested in the case $k=2$.
 We first give a constructive lower bound. We then give a near-optimal lower bound in the regime where the number of blocks is fixed and the block size tends to infinity. In other words, $n$ is fixed and $m$ goes to infinity.
 First, we provide a construction for $m=2$. In this case, we aim for the lower bound $\operatorname{MaxM}(n,2,2)$ to be quadratic in $n$.
 \begin{prop}\label{prop:maxm-p22}
  Let $p>2$ be a prime. Then $\operatorname{MaxM}(p,2,2)\ge\binom{p}{2}$.
 \end{prop}
\begin{proof}
  We assign to each of the $2p$ vertices an element of $\mathbb{Z}_p\times \mathbb{Z}_2$.
 Let
 \[
 S:=\left\{1,2,\dots,\frac{p-1}{2}\right\}.
 \]
 For each $(a,b)\in S\times\mathbb{Z}_p$, construct the pairing
  \[
  \mathcal{P}_{a,b}=\bigl\{\{(x,0),(ax+b,1)\}:x\in\mathbb{Z}_p\bigr\}.
  \]
 Since $a\neq0$, the map $x\mapsto ax+b$ is a bijection, so $\mathcal{P}_{a,b}$ is indeed a pairing. The constructed family has size $|S|p=\binom{p}{2}$.

 It remains to prove that these pairings are conflictless. Suppose for contradiction that two distinct pairs $(a,b),(c,d)\in S\times\mathbb{Z}_p$ produce a conflict. Then for some distinct $x,y\in\mathbb{Z}_p$ we have
 \[
 \{(x,0),(ax+b,1),(y,0),(ay+b,1)\}
 =
 \{(x,0),(cx+d,1),(y,0),(cy+d,1)\},
 \]
 after possibly swapping the two pairs in the second partition. Hence either
 \[
 ax+b=cx+d\quad\text{and}\quad ay+b=cy+d,
 \]
 or
 \[
 ax+b=cy+d\quad\text{and}\quad ay+b=cx+d.
 \]
 In the first case, subtracting gives $(a-c)(x-y)=0$, so $a=c$, and then $b=d$, a contradiction. In the second case, subtracting gives $(a+c)(x-y)=0$. Since $x\neq y$ and $a+c\not\equiv0\pmod p$ for $a,c\in S$, this is impossible. Thus the family is conflictless.
 \end{proof}
By taking several independent layers of the same construction, we obtain the following result.
 \begin{thm}\label{thm:maxm-prime}
  Let $p>2$ be a prime and let $m$ be a positive integer. Then $\operatorname{MaxM}(p,m,2)\ge\binom{p}{2}^{m-1}$.
 \end{thm}
\begin{proof}
Assign each of the $mp$ vertices an element of $\mathbb{Z}_p\times\mathbb{Z}_m$.
 Let
 \[
 S:=\left\{1,2,\dots,\frac{p-1}{2}\right\},
 \qquad p_{a,b}(x):=ax+b
 \]
 for $(a,b)\in S\times\mathbb{Z}_p$.
 For any
\[
\mathbf{ab}:=((a_1,b_1),(a_2,b_2),\dots,(a_{m-1},b_{m-1}))
\in (S\times \mathbb{Z}_p)^{m-1},
\]
construct the partition
\[
\mathcal{P}_{\mathbf{ab}}
=
\left\{
\{(x,0),(p_{a_1,b_1}(x),1),\dots,(p_{a_{m-1},b_{m-1}}(x),m-1)\}:x\in\mathbb{Z}_p
\right\}.
\]
Each layer map $p_{a_i,b_i}$ is a bijection, so this is a partition of $\mathbb{Z}_p\times\mathbb{Z}_m$ into $p$ sets of size $m$. There are $\binom{p}{2}^{m-1}$ such partitions.

We claim that this family is conflictless. Suppose that two unions of two blocks, one from $\mathcal{P}_{\mathbf{ab}}$ and one from $\mathcal{P}_{\mathbf{cd}}$, are equal. Comparing the zeroth layer, we may write these two unions using the same pair $x,y\in\mathbb{Z}_p$ with $x\neq y$, possibly in the opposite order. For each $i\in[m-1]$, comparison of the $i$th layer gives either
\[
a_i x+b_i=c_i x+d_i\quad\text{and}\quad a_i y+b_i=c_i y+d_i,
\]
or
\[
a_i x+b_i=c_i y+d_i\quad\text{and}\quad a_i y+b_i=c_i x+d_i.
\]
The first alternative gives $(a_i,b_i)=(c_i,d_i)$, while the second would imply $(a_i+c_i)(x-y)=0$, impossible because $a_i,c_i\in S$. Hence $(a_i,b_i)=(c_i,d_i)$ for every $i$, so $\mathbf{ab}=\mathbf{cd}$. Therefore distinct partitions do not conflict.
 \end{proof}
Note that the gap between this lower bound and the upper bound increases with the value of $m$.

For the case where $n$ is fixed and $m$ goes to infinity, we will construct an auxiliary hypergraph. We will be particularly interested in the matchings in this hypergraph, which will correspond to some conflictless families .We shall use the following form of the Pippenger-Spencer theorem.
\begin{thm}[Pippenger--Spencer~\cite{PIPPENGER198924}]\label{thm:pippenger-spencer}
For every integer $r\ge2$ and every $\epsilon>0$, there exists $\mu=\mu(r,\epsilon)>0$ such that the following holds. If $\mathcal H$ is an $r$-uniform hypergraph with maximum degree $\Delta$, minimum degree at least $(1-\mu)\Delta$, and maximum codegree at most $\mu\Delta$, then $\mathcal H$ has a matching covering all but at most an $\epsilon$-proportion of its vertices.
\end{thm}

\begin{thm}\label{thm:maxm-fixed-n-large-m}
For every fixed integer $n\ge5$ and every $\epsilon>0$, there exists $m_0=m_0(n,\epsilon)$ such that for all $m\ge m_0$,
\[
\operatorname{MaxM}(n,m,2)
\ge
(1-\epsilon)\frac{\binom{nm}{2m}}{\binom n2}.
\]
\end{thm}
\begin{proof}
Fix $n\ge5$ and $\epsilon>0$. We construct an auxiliary hypergraph $\mathcal H=\mathcal H_{n,m}$. Its vertex set is
\[
V(\mathcal H)=\binom{[nm]}{2m}.
\]
For every partition $\mathcal P$ of $[nm]$ into $n$ parts of size $m$, define an edge
\[
e_{\mathcal P}:=\{A\cup B:A,B\in\mathcal P,\ A\neq B\}.
\]
Thus $\mathcal H$ is $r$-uniform with $r=\binom n2$. A matching in $\mathcal H$ is precisely a conflictless family of partitions in the case $k=2$ of Question~\ref{q:partition}.

The hypergraph $\mathcal H$ is regular by symmetry. Its degree is
\[
\Delta
=
\frac12\binom{2m}{m}
\frac{((n-2)m)!}{(m!)^{n-2}(n-2)!}.
\]
Indeed, to count the edges containing a fixed $2m$-set $U$, we split $U$ into two unordered $m$-sets and then partition the remaining $(n-2)m$ vertices into $n-2$ unordered $m$-sets.

We now show that the maximum codegree is $o(\Delta)$ as $m\to\infty$, with $n$ fixed. Let $U,W\in\binom{[nm]}{2m}$ be distinct. If $U$ and $W$ both occur in some edge $e_{\mathcal P}$, then $|U\cap W|$ is either $m$ or $0$.

First suppose $|U\cap W|=m$. Then the three blocks $U\cap W$, $U\setminus W$, and $W\setminus U$ are forced. Hence the number of partitions containing both $U$ and $W$ is at most
\[
\frac{((n-3)m)!}{(m!)^{n-3}(n-3)!}.
\]
Dividing by $\Delta$ gives
\[
O_n\!\left(
\frac{m!\,((n-3)m)!}{\binom{2m}{m}((n-2)m)!}
\right)=o(1),
\]
where the notation $O_n(\cdot)$ allows constants depending only on $n$.

Now suppose $U\cap W=\emptyset$. Then we may first split $U$ into two $m$-sets and split $W$ into two $m$-sets, and then partition the remaining $(n-4)m$ vertices. Thus the codegree is at most
\[
\frac14\binom{2m}{m}^2
\frac{((n-4)m)!}{(m!)^{n-4}(n-4)!}.
\]
After division by $\Delta$, this is
\[
O_n\!\left(
\binom{2m}{m}\,m!^2\frac{((n-4)m)!}{((n-2)m)!}
\right)
=
O_n\!\left(
(2m)!\frac{((n-4)m)!}{((n-2)m)!}
\right)=o(1),
\]
because $n\ge5$ is fixed. Therefore $\Delta_2(\mathcal H)=o(\Delta)$.

Since $r=\binom n2$ is fixed, Theorem~\ref{thm:pippenger-spencer} applies for all sufficiently large $m$. Therefore $\mathcal H$ has a matching covering all but at most an $\epsilon$-proportion of its vertices. Each edge of $\mathcal H$ has size $\binom n2$, so this matching has size at least
\[
(1-\epsilon)\frac{|V(\mathcal H)|}{\binom n2}
=
(1-\epsilon)\frac{\binom{nm}{2m}}{\binom n2}.
\]
This matching is a conflictless family of partitions, proving the desired lower bound for $\operatorname{MaxM}(n,m,2)$.
\end{proof}

\begin{rem}\label{rem:nibble-gap}
Theorem~\ref{thm:maxm-fixed-n-large-m} should be read as a fixed-$n$, large-$m$ result. If instead $m$ is fixed and $n\to\infty$, then the auxiliary hypergraph above has growing uniformity $\binom n2$. The fixed-uniformity form of Theorem~\ref{thm:pippenger-spencer} is then not sufficient without quantitative control of the parameter $\mu(r,\epsilon)$, or without using a nibble theorem stated directly for growing uniformity.
\end{rem}

We can translate the conflictless-partition construction back into the language of the MMS property, in the way it was already outlined.
 \begin{lemma}\label{lem:blowout-union-general}
Let $\mathcal{P}_1,\dots,\mathcal{P}_\ell$ be a conflictless family of partitions of $[mn]$ into $m$-tuples in the case $k=2$ of Question~\ref{q:partition}. Suppose that, for each $i\in[\ell]$, there exists a $d_i$-regular graph $G_i$ on $n$ vertices with the MMS property, where $d_i\ge0$. Then there exists a $2m$-uniform $\left(\sum_{i=1}^\ell d_i\right)$-regular hypergraph on $mn$ vertices with the MMS property.
 \end{lemma}
\begin{proof}
For each $i$, identify the vertices of $G_i$ with the $m$-tuples in $\mathcal{P}_i$ and form the corresponding blowout $F_i$ on the common vertex set $[mn]$. By Lemma~\ref{lem:blowout}, each $F_i$ is $2m$-uniform, $d_i$-regular, and has the MMS property. Since the partitions are conflictless for pairs of blocks, no edge of $F_i$ is also an edge of $F_j$ for $i\neq j$. Hence, the $F_i$ are pairwise edge-disjoint, and Lemma~\ref{lem:edge-disjoint-union} gives the result.
\end{proof}

 \begin{cor}\label{cor:prime-blowout}
  Let $p\ge 11$ be a prime and let $m$ be a positive integer. For every even integer
  \[
  4\le d\le (p-1)\binom{p}{2}^{m-1},
  \]
  there exists a $2m$-uniform $d$-regular hypergraph with the MMS property on $mp$ vertices.
 \end{cor}
\begin{proof}
By Theorem~\ref{thm:maxm-prime}, there is a conflictless family of
\(
L=\binom{p}{2}^{m-1}
\)
partitions of $[mp]$ into $m$-tuples. Since $p\ge11$, Theorem~\ref{thm:regular-existence} applies to graphs on $p$ vertices and gives a $q$-regular graph with the MMS property for every even $q$ with $4\le q\le p-1$. We also allow $q=0$, using the empty graph.

It remains only to express $d$ as a sum of $L$ terms, each belonging to $\{0,4,6,\dots,p-1\}$. Put $M=(p-1)/2$ and $s=d/2$. We need to write $s$ as a sum of at most $L$ terms from $\{2,3,\dots,M\}$. This is immediate if $s\le M$. If $s>M$, write $s=qM+r$ with $0\le r<M$. If $r=0$, use $q$ terms equal to $M$; if $r\ge2$, use $q$ terms equal to $M$ and one term equal to $r$; if $r=1$, replace one term $M$ and the remainder $1$ by two terms $M-1$ and $2$. In the last case $q+1\le L$ because $s\le LM$ and $r=1$. Multiplying these terms by $2$ gives the desired representation of $d$. Now apply Lemma~\ref{lem:blowout-union-general} with these graph degrees.
\end{proof}

\begin{cor}\label{cor:fixed-n-near-optimal-blowout}
Let \(n\) be a fixed odd integer with \(\phi(n)\ge 8\), and let
\(\epsilon>0\). Then there exists \(m_0=m_0(n,\epsilon)\) such that for
every \(m\ge m_0\) and every even integer
\[
4\le d\le (1-\epsilon)(n-1)\frac{\binom{nm}{2m}}{\binom n2},
\]
there exists a \(2m\)-uniform \(d\)-regular hypergraph with the MMS property
on \(nm\) vertices.
\end{cor}
\begin{proof}
Put
\(
U_m:=\frac{\binom{nm}{2m}}{\binom n2}.
\)
Apply Theorem~\ref{thm:maxm-fixed-n-large-m} with $\epsilon/2$ in place of $\epsilon$. For all sufficiently large $m$, there is a conflictless family of partitions of $[nm]$ into $m$-tuples of size at least
\[
L:=\left\lfloor (1-\epsilon/2)U_m\right\rfloor.
\]
Increasing $m_0$ if necessary, we may assume that $L\ge (1-\epsilon)U_m$.

Since \(\phi(n)\ge 8\), Theorem~\ref{thm:regular-existence} and Proposition~\ref{prop:extremal-degree} gives a $q$-regular graph with the MMS property on $n$ vertices for every even $q$ with $4\le q\le \phi(n)$ and $\lceil\frac{n}{2}\rceil+2\leq q\le (n-1)$; we again allow $q=0$ by using the empty graph. The assumed upper bound on $d$ gives $d\le (n-1)L$. As in the proof of Corollary~\ref{cor:prime-blowout}, every even $d\le (n-1)L$ can be written as a sum of $L$ terms from $$\{0,4,6,\dots,\phi(n)\}\cup\{\lceil\frac{n}{2}\rceil+2,\dots (n-1)\}.$$ Applying Lemma~\ref{lem:blowout-union-general} to these graph degrees gives the required $2m$-uniform $d$-regular MMS hypergraph.
\end{proof}

\bibliographystyle{plainnat}
\bibliography{bibliography}

\end{document}